\documentclass[11pt]{amsart}
\usepackage[all]{xy}
\usepackage{graphicx}
\usepackage[colorlinks=true, urlcolor=blue,bookmarks=true,bookmarksopen=true,
citecolor=blue]{hyperref}

\newcommand{\R}{{\mathbb{R}}}

\newcommand{\Z}{{\mathbb{Z}}}

\newcommand{\p}{{\partial}}

\newcommand{\om}{{\omega}}
\newcommand{\eps}{{\varepsilon}}

\newcommand{\ga}{{\gamma}}

\newcommand{\Ham}{{\rm Ham}}

\newcommand{\Si}{{\Sigma}}

\newcommand{\QED}{\hfill$\Box$\medskip}
\newtheorem{theorem}{Theorem}
\newtheorem{corollary}{Corollary}[section]
\newtheorem{definition}{Definition}

\newtheorem{lemma}[corollary]{Lemma}

\newtheorem{prop}[corollary]{Proposition}
\newtheorem{proposition}[corollary]{Proposition}

\begin{document}

\title{On the injectivity radius in Hofer's Geometry}

\author{Fran\c{c}ois Lalonde}
\address{D\'epartement de math\'ematiques et de Statistique, Universit\'e de Montr\'eal, C.P. 6128, Succ. Centre-ville, Montr\'eal H3C 3J7, Qu\'ebec, Canada}
\email{lalonde@dms.umontreal.ca}

\author{Yasha Savelyev}
\address{Centre de Recherches Math\'ematiques, Universit\'e de Montr\'eal, C.P. 6128, Succ. Centre-ville, Montr\'eal H3C 3J7, Qu\'ebec, Canada}
\email{savelyev@crm.umontreal.ca}

\maketitle

\noindent
{\bf Abstract.} \,
{\em In this note we consider the following conjecture: given any closed symplectic manifold $M$, there is a sufficiently small real positive number $\rho$ such that the open ball of radius $\rho$  in the Hofer metric  centered at the identity on the group of Hamiltonian diffeomorphisms of $M$ is contractible, where the retraction takes place in that ball -- this is the strong version of the conjecture -- or inside the ambient group of Hamiltonian diffeomorphisms of $M$ -- this is the weak version of the conjecture. We prove several results that support that weak form of the conjecture.
\footnote{2010 Mathematics Subject Classification 53C15, 53D12, 53D40, 53D45, 57R58, 57S05, 58B20.}
\footnote{Hofer's geometry, Lagrangian submanifolds, Quantum characteristic classes.}
}

\bigskip \bigskip 
\section{General facts and results} 

     Consider a closed symplectic manifold $(M, \om)$ of any dimension. We recall that the Hofer norm on the group  $\Ham(M)$  of Hamiltonian diffeomorphisms of $M$ 
     assigns to each diffeomorphism $\phi \in \Ham(M)$ the infimum, over all Hamiltonians $H: M \times [0,1] \to \R$ 
     whose time-one flow equals $\phi$, of the mean total variation of $H$ defined by
     $$
      \int_0^1 (\max_M H_t - \min_M H_t) dt.
      $$ 
      Given a real number $\rho \geq 0$, let us now denote by $B_H(\rho)$ the
     subspace of $\Ham(M)$ of all diffeomorphisms of Hofer norm smaller or equal to $\rho$.  
     What is the topology of $B_H(\rho)$ when $\rho$ goes to zero, or when $\rho$ goes to $\infty$ ?
 When the manifold is a surface of genus larger than $0$, it has been proved by Lalonde and McDuff \cite{LM} that the group $\Ham(M)$ 
 has infinite diameter, i.e that $B_H(\rho)$ does not contain $\Ham(M)$ whatever the large value of $\rho$ chosen. This was extended by Lalonde and Pestieau  \cite{LP} to manifolds of the form $\Sigma \times M$ for $M$ weakly exact and $\Si$ the same kind of surface. The proof of the unboundedness of the group of Hamiltonian diffeomorphisms of the 2-sphere was given by Polterovich \cite{P}.
 A natural conjecture is that, for $\rho$ small enough,  $B_H(\rho)$ is
 contractible. Note that the corresponding statement is always true for finite
 dimensional Finsler manifolds $X$, as the exponential map at $x$ is always
 defined and is a diffeomorphism on a sufficiently small neighborhood of $ 0 \in
 T _{x} X$, see for example \cite[Chapter 11]{Shen}. Moreover the analogous statement holds for the
 group of volume preserving diffeomorphisms  of a Riemannian
 manifold $X$, with its natural $L ^{2}$-metric, see Ebin-Marsden \cite{Ebin}.
 In this latter case, however, it is a deep and difficult fact, although it is again essenrtially a
 statement about the existence of an exponential map in a neighborhood of the tangent
 bundle at a point. In the finite dimensional Finsler setting, the size of the
 largest ball in $T _{x} X$ on which the exponential map is defined and is a diffeomorphism is
 called the injectivity radius at $x$. 
 For a general metric space $X$, we may call the supremum
 of $\rho$'s for which the $ \rho$-ball around $x \in X$ is contractible, 
 the {\it injectivity radius} at $x$, although in the finite dimensional
 Finsler setting, the classical injectivity radius is only a lower bound for
 the above generalization. Summarizing, if the conjecture holds, we have an interesting
 numerical invariant of a symplectic manifold $(M, \omega)$: the injectivity radius of $ \text {Ham}(M, \omega)$.
 
 We may also ask if  for $\rho$ small enough, the inclusion map of $B_H(\rho)$
 into $\Ham(M, \om)$ is null-homotopic (with respect to the
 $C^{\infty}$-topology, and therefore with respect to the Hofer topology). We will refer to this as the weak conjecture since the contraction to a point of the ball $B_H(\rho)$ may then take place in the full ambient space $\Ham(M, \om)$, instead of the ball itself. It is perhaps worth noting that the terminology {\it ball} for $B_H(\rho)$ might be misleading: $B_H(\rho)$ is not parametrized by a ball, it is a {\it subset} of 
 $\Ham(M, \om)$ whose topology might be a priori complicated.
  Note that for surfaces of genus $g \geq 1$, this weak conjecture is obvious
 because the whole group $\Ham(M,\om)$ is contractible in that case.
  
 Another related question, and indeed a possible way to approach the above
conjecture, is to show that there is a  $\rho>0$ such that the space of paths from the
identity to $x\in B _{H} (\rho)$, minimizing the Hofer length up to $\delta$, is
contractible for some $ \delta$. From now on, such paths will be referred to as
$\delta$-minimizing. Denote the latter path space $P 
_{\delta} (id, x)$. We may in general ask for which $\rho$ and which $ (M,
\omega)$ is the inclusion $P  _{\delta}(id, x) \to P (id, x)$ null homotopic. 
Interestingly while this may seem like a harder question than the original
conjecture, the theory of Gromov-Witten invariants in particular quantum
characteristic classes \cite{GS}, give  a partial answer, which we now describe.
For the moment, let us merely consider quantum classes as certain  
invariants of homotopy groups: $$a \mapsto qc _{k-1} (a) \in QH (M)),$$ for
$$a \in \pi _{k} \text {Ham}(M, \omega) \simeq \pi _{k-1} \Omega \text
{Ham}(M, \omega),$$ 
(the shift by $-1$ is for consistency with \cite{GS}). When $k=1$ we just get
the Seidel invariant: $$qc _{0} (a) = S (a), $$
see \cite{citeSeidel$pi_1$ofsymplecticautomorphismgroupsandinvertiblesinquantumhomologyrings}.
\begin{definition}
   We say that
   \emph {\textbf{quantum classes detect rational homotopy groups}} of $ \text
{Ham}(M, \omega)$ if whatever $k$ and an element $a \in \pi _{k} (
\text {Ham}(M, \omega), \mathbb{Q})$ given, it vanishes as soon as $qc _{k-1} (a)$ vanishes. \end{definition} 
This is known to hold for example for $M=S ^{2}$, and
$M= \mathbb{CP} ^{2}$  as in this case the Hamiltonian group
retracts onto the compact  subgroups $PSU  (2)$, respectively $ PSU
(3)$ by a classical theorem of Smale, respectively classical theorem of Gromov,
\cite{citeGromovPseudoholomorphiccurvesinsymplecticmanifolds.}. The non-zero rational homotopy groups are in degrees $3$, respectively
$3,5$. These degrees are in the so called stable range: $0 \leq k-1 \leq
2n-2$, in the sense of
\cite{citeSavelyevBottperiodicityandstablequantumclasses.},
of rational homotopy groups of $PSU (n)$. The assertion that quantum classes
detect homotopy groups of $PSU (2),
PSU (3)$ then immediately follows from the main theorem of
\cite{citeSavelyevBottperiodicityandstablequantumclasses.}. 

Before going further on, we mention that we will follow the following
conventions: the homology is always over $ \mathbb{Q}$ unless specified
otherwise, and the quantum homology of a monotone symplectic manifold is also taken with
$\mathbb{Q}$ coefficients and with $\mathbb{Z}_{2}$ grading.

Let $(M, \omega)$ be a monotone symplectic manifold $\omega = c \cdot c_1 (TM)$,
with monotonicty constant $c>0$. Set $\hbar= min (c \cdot N, D (M))$  
where $N$ is the minimal positive Chern number $ \langle c_1 (TM), [u] \rangle
$ over all $u$ of $u ^{*} TM $ for $u: S ^{2}
\to M $, and where $D (M)$  is the
infinum over the positive Hofer length of non-contractible loops in
$\text{Ham} (M,\omega)$. If the above Chern numbers all vanish, set $\hbar
= D (M)$. If $\pi _{1} \text{Ham}(M, \omega)=0$, set $D (M)=\infty$.

\begin{theorem} Suppose we are given a  monotone symplectic manifold $ (M, \omega)$, for which
quantum classes detect rational homotopy groups, then $\hbar >0$ and the
inclusion $i: P _{\delta/3, +} (id, x)
\to P (id,x)$ vanishes on rational homotopy groups for $x \in B
(\hbar/2-\delta)$, for all $\delta >0$.
\end{theorem} 
\begin{proof}  Here $P _{\delta/3,+}(id, x) $ denotes the space of paths minimizing the
   positive Hofer length functional: 
   \begin{equation*}
     L ^{+}(\gamma)=\int _{0}^{1} \max _{M}  H ^{\gamma}_{t}     dt,
   \end{equation*}
  up to $\delta /3$, where $H ^{\gamma} $ is the  generating function for $\gamma$
  normalized to have zero mean at each moment.
   Fix a
$\delta/3$-minimizing $p _{0} \in P (id, x)$. Given $f: S ^{k} \to P _{\delta/3,
+}
(id, x) $, we get a map $ \widetilde{f}: S ^{k} \to \Omega \text {Ham}(M, \omega)$, 
$\widetilde{f} (s) = f(s) \cdot p _{0} ^{-1} $, for $ \cdot$  the
concatenation product. Clearly the length of each loop $ \widetilde{f} (s)$ is
less than $\hbar$. But then by the proof of \cite[Lemma
3.2]{citeSavelyevVirtualMorsetheoryon$Omega$Ham$(Momega)$.}, $qc
_{k} ([ \widetilde{f}])$ vanishes. Let us explain this. The invariant $qc
_{k} ([ \widetilde{f}])$ is defined by counting pairs $ (u, s)$ 
 for $u$ a $J _{s}$-holomorphic section with some constraints, of the bundle $M
 \hookrightarrow X _{s} \to \mathbb{CP} ^{1}$, obtained by using $\widetilde{f} _{s}$ as a
clutching loop: $$X _{s} =
M \times D ^{2} \sqcup _{\widetilde{f}_s}  M \times D ^{2}, $$
 where $J_{s}$ is tamed by a symplectic form $\Omega_s$ on $X _{s}$, with both
 of these smoothly varying. Now $(u,s)$ can contribute to the invariant only if
 $$ \langle c_1 (T
^{vert} X _{s}), [u] \rangle <0,$$ for dimensional reasons that one can check easily. By
assumption, each $\widetilde{f}_{s}  $ is contractible and so 
as a smooth bundle $X _{s} \simeq M \times S^2$ this means that $(u,s)$ can contribute only if
$\langle c_1 (T ^{vert} X _{s}), [u] \rangle < - N$, so that $$ \langle \omega,
[u] \rangle < - \hbar,$$ where $\omega$ is the natural  form on $X _{f _{s}}$ under the identification:  $X _{s}
\simeq M \times  \mathbb{CP} ^{1}$. This is the identification induced by any chosen
contraction of $\widetilde{f}_{s}  $. The result is independent of the
identification as the Hamiltonian gauge group of $M \times \mathbb{CP}^{1} $ acts trivially on homotopy groups,
see \cite{citeKedraMcDuffHomotopypropertiesofHamiltoniangroupactions}. We can
also make this point more transparent by using coupling forms, but we avoid
introducing extra technology at this point.
Finally \cite[Lemma
3.2] {citeSavelyevVirtualMorsetheoryon$Omega$Ham$(Momega)$.} tells us that the length of the loop $\widetilde{f} (s)$ must
then be at least $ \hbar$.  
 
So we conclude that  $ \widetilde{f}$ is vanishing on rational homotopy groups,
but then clearly the same must hold for $i \circ f$. 
\end{proof}
The question of the injectivity radius can also be developed as follows. Given any
class $\alpha \in H_*(\Ham(M), \Z)$, define the { \it higher $\alpha$-capacity} of $(M, \om)$ as the infimum of $\rho$ such that there exists a class $\xi \in H_*(B_H(\rho), \Z)$ with 
   $\iota_*(\xi) = \alpha$.  Here $\iota$ is the injection of $B_H(\rho)$ inside $\Ham(M)$.  Given any such non-zero $\alpha$,
    the conjecture, if true, shows that the $\alpha$-capacity is not zero. On the other hand, the following proposition shows that it is bounded above. 
    
  \begin{prop} Given any closed symplectic manifold $(M, \om)$ and any class $\alpha \in  H_*(\Ham(M), \Z)$, there is $\rho \geq 0$ such that $\alpha$ is realized inside $B_H(\rho)$.
  \end{prop}
  
  \proof  The image of a cycle $\alpha$ is a compact subset $K$ of $\Ham(M)$.
  Suppose that $\phi$ belongs to $K$, with Hofer norm of $\phi$ denoted $E(\phi)$.
  Consider the ball of radius $\eps$ in Hofer's norm centered at $\phi$. It contains an open set $U(\phi)$ centered at $\phi$ in the $C^{\infty}$-topology because the $C^{\infty}$ topology is finer than the Hofer topology. By the triangle inequality, the elements of $U(\phi)$ have Hofer norm at most $E(\phi) + \eps$. Because $K$ is compact, there is a finite collection of these open sets. \QED

   Thus higher capacities belong to $(0, \infty)$ if the conjecture is true. 
   For $M = S^2$, we show that:
   
   \begin{prop} There is $\rho$ small enough so that the injection $B_H(\rho) \to \Ham(S^2)$ does not catch any of the $\Z$-generators of the $\Z$-homology of $ \Ham(S^2)$.
 \end{prop}
   
   \proof By remarks following Theorem \ref{Thm:simply-connected} 
    there is a $\rho_0>0$ such that all loops in $B_H(\rho_0)$ are contractible
    inside $\Ham(S^2)$. Thus  the generator $\ga$ of the fundamental group of
    $\Ham(S^2)$  is not homologous inside $\Ham(S^2)$ to a 1-cycle lying in
    $B_H(\rho_0)$ since otherwise, the concatenation of the chain realizing this
    homology with the disc realizing the homotopy to a point would give a chain
    in $\Ham(S^2) \simeq SO(3)$ realizing a homology between $\ga$ and a point,
    a contradiction. This proves our statement for the generator of the first
    cohomology group of $\Ham(S^2)$.
   
      Now suppose by contradiction that there is a singular 4-chain $c$ over the
      integers in $\Ham(S^2)$  such that $\p c = [SO(3)] - d$ where $d \in
      Z_3(B_H(\rho_0))$, $Z_3$ being the $3$-cycles over the integers. Let
      $(K,f)$ be the realization of $c$, i.e let $K$ be a compact simplicial
      complex and $f: K \to \Ham(S^2)$ be a continuous map that realizes the
      homology $c$. Thus $K$ is collection of simplices with integral
      coefficients, and with the incidence relations dictated by $c$. One may
      realize $K$ as a piecewise linear complex with integral coefficients in
      some $R^N$, for $N$ large enough. The boundary of $K$ is by definition the
      sum over all simplices of top dimension of the boundary of each such
      simplex with the coefficient coming from the simplex. By definition of
      $K$, this boundary is equal to $K_0 - K_1$ where the subcomplex $K_0$ of
      $K$ is an abstract  singular triangulation of $[SO(3)]$ and the
      restriction $f_0$ of the map $f$ to $K_0$ identifies $K_0$ with a singular
      triangulation of $[SO(3)] \subset \Ham (S^2)$,  while the restriction
      $f_1$ of $f$ identifies $K_1$ with the cycle $d$. Now let's compose $f$
      with the retraction  $r: \Ham(S^2) \to SO(3)$. This gives a mag $g: K \to
      SO(3) \subset \Ham(S^2)$. Now, $SO(3)$ is a smooth manifold and hence
      there is a continuous map $g': K \to SO(3)$ $C^0$-close to $g$ and
      homotopic to $g$ such that the restriction of $g'$ to each simplex is a
      smooth map. Thus, if $\ga$ is a smooth curve of $SO(3)$ representing the
      generator of $\pi_1(SO(3))$ and transverse to all simplices of $(K, g')$,
      the inverse image $g'^{-1}(\ga)$ of $\ga$ by $g'$ is a $2$-dimensional
      subspace of $K$ that can be represented, up to smooth subdivision of $K$,
      by a smooth subcomplex $L \subset K$ with weights (coming from each
      simplex) with boundary equal to $\Gamma - \Gamma'$ where $\Gamma$ is a
      $1$-cycle inside $K_0$ mapped to $\ga$ by $g'$ and where $\Gamma'$ is a
      subset of $K_1$.
     
         Finally, note that $f$ maps  the cycle $\Gamma$ to the $1$-cycle $\ga$
         of $SO(3)$ and the cycle $\Gamma'$ inside  $B_H(\rho)$. But $\Gamma$
         and $\Gamma'$ are homologous inside $K$, thus their images by $f$ are
         homologous inside $\Ham(S^2)$. This means that there is a $1$-cycle
         inside $B_H(\rho)$ that is homologous to the generator $\ga$ of
         $\Ham(S^2)$, a contradiction.

  \QED

%
%
%
%
%
%

    \begin{theorem} \label{lemma.simply.con} Let $D$ be the infimum over 
length of all  essential loops in a length space $X,g$, and let
$ B _{x_0}$ be the $ \epsilon$-ball around $x _{0}$ in $ X,g$ with $ \epsilon =
D/2 -\delta$. Then the inclusion of $B _{x _{0}}$ into $X$ vanishes on $ \pi_1$,
for all $ \delta>0$.
\end{theorem} 
\begin{proof} Let $E$ be the space of paths in $X$ starting  at  a fixed $x_0
\in B$ and ending at some $ b \in B$, in the minimizing homotopy class relative to  endpoints, that is to say in the class $[p]$, so that $d(x,b)$ is the infimum
over length of all paths in  $ [p]$. By our assumptions and because the
shortest length of an essential loop is at least $D$, this class is uniquely
determined. For a loop $\gamma: S ^{1} \to
B$, consider the pullback  $E' = \gamma ^{*} E$ of $ E \to B$. We
show that $ E'$ has a section over $ S ^{1} = [0,1]/ \sim$. Set $ b_0 = \gamma (0)$ and let
 $p_0$ be any path from $ id$ to $ b_0$ minimizing  length up to $
\delta/3$, we will just say $ \delta/3$ minimizing from now on. In particular $
p_0$ is in the fiber of $ E$ over $b_0$ -- if it were in the wrong homotopy class
its length would be at least $D/2$. Partition $ S ^{1}$ into segments $ s _{i}$ so that the length 
of each segment $\gamma|{s_i}$ be at most $ \delta/3$. For simplicity say there are 2
segments. Then we have a canonical section of $
E'$ over $ s_0$, which is $ p_0$ over 0, and over $ t \in s_0$ it is $ p
(t)$ defined by concatenating $ p_0$ with $ \gamma| _{[0,t]}$. Since each $ p
(t)$ has length less than $D/2$ it is in the right class, since otherwise its 
length would be at least $D/2$. But since each $ p (t)$ is in the right class
we can change this section (homotopy extension property) so that over the right
end point $ t= 1/2$, $ p (1/2)$ again minimizes  length up to $ \delta/3$. Repeating this we get a
section of $ E'$ which is 2 valued only over $ 0$, but since the fiber is connected we 
can adjust it to be an actual section. Given this section, we can
contract $ \gamma$ by the associated family of paths.
\end{proof}  
As a corollary we have:
\begin{theorem} \label{Thm:simply-connected} Let $(M, \om)$ be a symplectic manifold such that 
there is a lower bound $D$ for the Hofer length of a non-contractible loop in
$\Ham(M)$. Then the inclusion of $B_H(\rho)$ into $\Ham(M)$ vanishes on $
\pi_1$ for $\rho=D/2 - \eps$ for all $\eps > 0$.
  \end{theorem}
  
  By Lalonde-McDuff results in \cite{LM}, any ruled symplectic $4$-manifold or any surface satisfies this hypothesis.
     
        Let us now consider the conjecture stating  that  the space $B_H(\rho)$ of all Hamiltonian diffeomorphisms of $S^2$ of Hofer norm less or equal to $\rho$ is contractible inside $\Ham(M)$, for $\rho$ small enough (weak conjecture). Let us denote by $L(\rho)$ the topological space of all images of the standard oriented equator $L \subset S^2$ by  Hamiltonian diffeomorphisms of Hofer's norm less or equal to $\rho$. So $L(\rho)$ is included in the space $L(\infty)$ of smooth oriented embedded loops that divide the sphere into two regions of equal areas.
        
        \begin{proposition} If $L(\rho)$ is contractible in $L(\infty)$, then
        the space $B_H(\rho)$ is also contractible in $\Ham(M)$.
        \end{proposition}
        
        \proof
        Consider the map $B_H(\rho) \to L(\rho)$ that assigns to each
        diffeomorphism $\phi$ the image under $\phi$ of the standard oriented
        equator in $S^2$. It is not hard to verify that 
        this is a Serre fibration. And this is a sub fibration of the Serre fibration $\Ham(M) \to L(\infty)$. The fiber of this sub fibration is the space of all Hamiltonian diffeomorphisms of Hofer norm less or equal to $\rho$ that preserve, not necessarily pointwise, the equator. 
        It is well-known that the same space, but with no restriction on the energy, is homotopy equivalent to $S^1$ since it retracts to the space of rotations of the closed disk. 
       By our Theorem \ref{Thm:simply-connected},  the inclusion of the fiber of $B_H(\rho) \to L(\rho)$  inside the total space $\Ham(M)$ is contractible inside $\Ham(M)$. Now use the fact that the base $L(\rho)$ is contractible inside $L(\infty)$ to retract each of the fibers of $B_H(\rho) \to L(\rho)$ onto the fiber at a point $L' \in L(\rho) \subset L(\infty)$. This retraction takes place in $\Ham(M)$. Then compose with the retraction of that fiber to a point inside $\Ham(M)$. 
        \QED
        
         So the problem of the contractibility of $B_H(\rho)$ inside $\Ham(M)$ reduces to the problem of the contractibility of $L(\rho)$ inside $L(\infty)$.
        
        \section{The space $L(\rho)$ and the double octopus}
        
    The main question is:  is it possible to find a $2$-cycle inside $L(\rho)$, 
       for arbitrarily small $\rho$, homologous in $L(\infty)$ to the
       $S^2$-cycle made of linear Lagrangians ?
      
      One way of approching that question is to look for the obstructions in designing an algorithm that would retract all exact Lagrangians sufficiently close to the standard oriented equator $L$ to $L$. This is what we do in this section. 
       
    Each $L' \in L(\rho)$ comes with an orientation that defines two sides $H_+(L')$ and $H_-(L')$ (H for ``hemisphere''). Here  $H_+(L)$ and $H_-(L)$ are the standard upper and lower hemispheres.  Set: 
        
        $$
        R_+  = H_-(L') \cap H_+(L)  $$
        
     $$   R_-  =  H_+(L') \cap H_-(L)   $$
        
      $$  G_+ =  H_+(L') \cap H_+(L)  $$
        
      $$  G_- =  H_-(L') \cap H_-(L)
        $$
          
 \begin{lemma} Each connected component $R$ of $R_+$ or $R_-$ has area bounded above by $\rho$.
 \end{lemma}
 
     \proof  Lift $L'$ to some Hamiltonian diffeomorphism $\phi$ of energy less or equal to $\rho$. Then $\phi^{-1}$ sends each connected component $R$ of $R_+$ to the lower standard hemisphere. But such a connected component lives in the upper standard hemisphere, and so $\phi^{-1}$ displaces it. Thus, by the energy-capacity inequality, $\rho$ is greater or equal to the area of $R$. The same argument applies as well if $R$ is a component of $R_-$. 
     \QED
     
        In order to prove that $L(\rho)$ is contractible in $L(\infty)$, one would like to define an algorithm that retracts $L'$ to $L$ inside $L(\infty)$ in a canonical way, i.e in a way that depends continuously on $L'$.    Note that $L(\infty)$ contains the space $LL$ of linear Lagrangians, i.e the one consisting of all oriented great circles. This space is identified to $S^2$ in the obvious way, with the south pole of $LL$ being the standard equator $L$ and the north pole being the same equator with the opposite orientation $L^{opp}$. On $LL -  \{L^{opp}\}$, there is a retraction to $L$ given by reducing simultaneously  the areas of the two 2-gones $R_+$ and $R_-$.  Of course, the continuity of this argument breaks down at $L^{opp}$ since the direction of the retraction depends on the slight perturbation of $L^{opp}$ inside $LL$ that one chooses. Note of course that, near $L^{opp}$, the regions $R_+$ and $R_-$ are big and such configurations cannot appear in $L(\rho)$ for small enough $\rho$.  
        
        However one can first slightly perturb $L^{opp}$ so that it be the graph $L'$ of a small sinusoidal function of $L$.  So, if the number of points in $L \cap L'$ is $2k$, the 2-sphere decomposes into a large $2k$-gone $R_+$ in the upper hemisphere, a large $2k$-gone $R_-$ in the lower hemisphere, and a sequence of small alternating 2-gones of $G_+$ and $G_-$. Note that $L'$ is in general position with respect to $L$. Now any reasonable algorithm would retract $L'$ to $L^{opp}$ and would break there.  However, by the energy-capacity inequality, such a configuration $L'$ cannot belong to $L(\rho)$. The last step is to start with $L'$ and inflate each $2$-gone of $G_+ \cup G_-$ in its own hemisphere in such a way that, at the end of this inflation, we get an element $L''$ of $L(\rho)$ which has both a $(\Z/k \Z)$-symmetry and a $\Z_2$-symmetry and is made of: a $2k$-gone in $R_+$ whose center is at the north pole, which is an arbitrarily small thickening of a star with $k$ branches, the end of each branch being on the equator; a $2k$-gone in $R_-$ whose center is at the south pole, which is an arbitrarily small thickening of a star with $k$ branches and such that the ends of the branches in  $R_-$ meet the equator at mid-points between the ends of branches of $R_+$; a sequence of large $2$-gones alternating between  $G_+$ and $G_-$. Each $2$-gone in $G_-$ can be viewed as the prolongation of a branch of $R_+$ as the branch crosses the equator, and similarly for $G_+$. This is what we call the {\em double octopus}. It is made of two octopuses, one based at the north pole and the other at the south pole, with very thin body and legs, both with large feet (a foot is the part of a leg that crosses the equator). 
        
        \includegraphics[scale=.5]{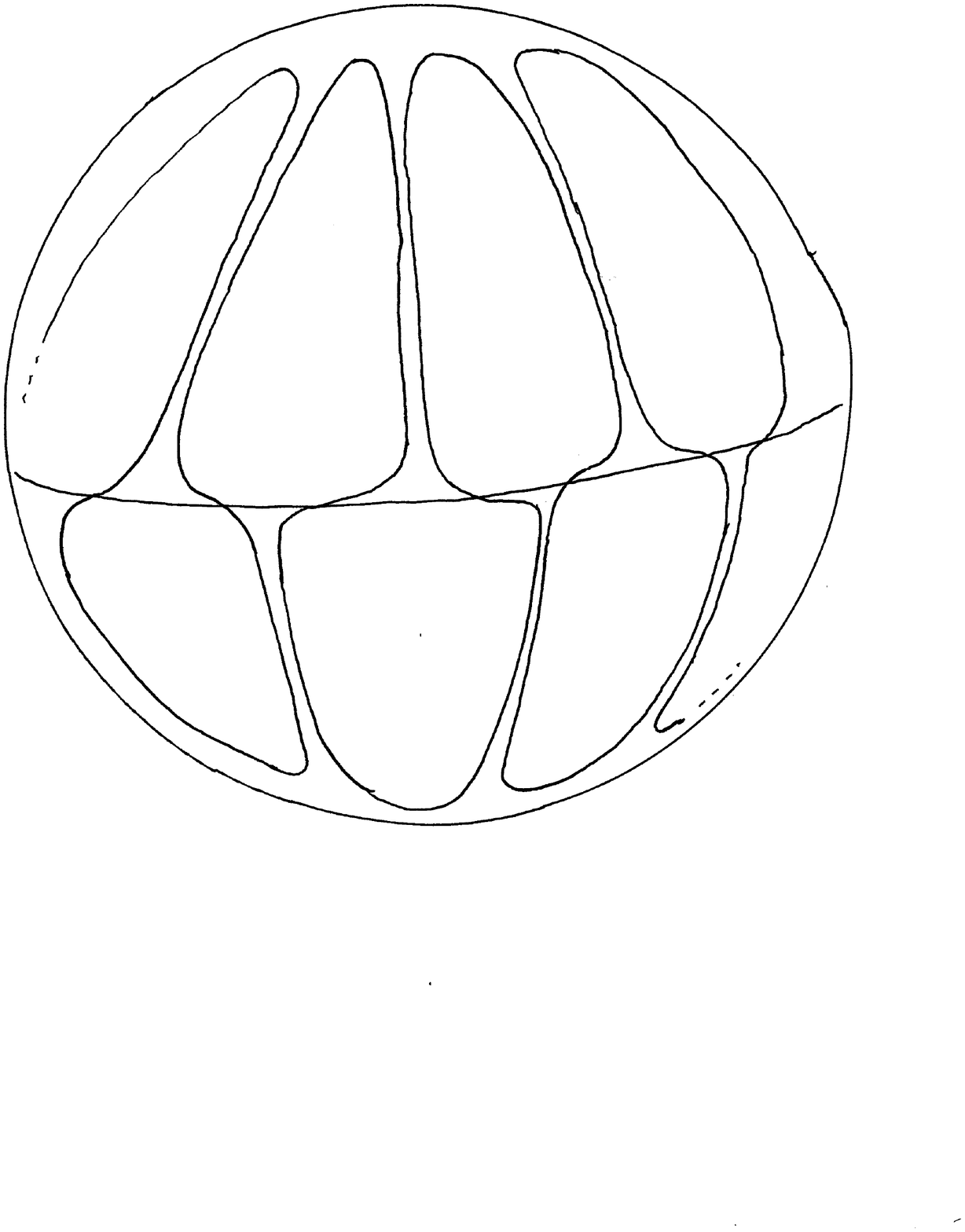}

   Denote by $\mathcal{O}_{k,a}$ this oriented exact Lagrangian of $S^2$: here $k$ is the number of legs of each of the two octopuses, and $a$ is the area of the intersection of the standard upper hemisphere with the upper  octopus (that is to say the one based at the North pole). Thus $k$ may be as large as we wish and $a$ as small as we wish.    Hence $\mathcal{O}_{k,a} \in L(\infty)$ has all the properties of an element of $L(\rho)$ for small $\rho$: both of its $R_+$ and $R_-$ are arbitrarily small.
 If $\mathcal{O}_{k,a}$ belongs to $L(\rho)$ for small $\rho$, there is not much hope to construct a retraction of $L(\rho)$ to $L$ inside $L(\infty)$. 
 
         Khanevsky and Zapolsky \cite{KZ} observed that actually the double octopus does not constitute a counter-example to the main conjecture of this paper. Here is their observation:

        \begin{proposition} For each small  enough $\rho$,  the double octopus configuration $\mathcal{O}_{k,a}$ (whatever the value $k \ge 2$ and for each $a$ small enough) does not lie in $L(\rho)$.
        \end{proposition}
        
        \noindent
        {\it Proof.} \;  The distance between $L$ and $L^{opp}$ is equal to the area of $S^{2}$, that is say to $4 \pi$.  Let $\rho$ be given. Consider $ \mathcal{O}_{k,a}$. Assume that it belongs to $L(\rho)$. The distance from $ \mathcal{O}_{k,a}$ to $L^{opp}$ is less or equal to the area of two consecutive legs, that is to say to $(4 \pi / k) - \eps$. Then, by the triangle inequality:
        
        $$
        d(L, \mathcal{O}_{k,a}) \geq  \| d(L, L^{opp}) - d(L^{opp}, \mathcal{O}_{k,a}) \| 
        $$
        \noindent
        so this means
        $$
        d(L, \mathcal{O}_{k,a}) \geq 4 \pi - 4 \pi/k +\eps  =  \frac{4\pi(k-1)}{k}  + \eps
        $$
        
        \noindent
        which shows that octopuses with $k \geq 2$ and $a$ sufficienttly small cannot lie in arbitrarily small balls around the standard equator $L$.

        \QED


        \end{document}